\documentclass{amsart}

\usepackage[margin=1.51in]{geometry}
\usepackage{amsmath}
\usepackage{amsthm}
\usepackage{amssymb}
\usepackage{amsfonts}
\usepackage{amsrefs}
\usepackage{color}
\usepackage{tikz}

\tikzstyle{shaded}=[fill=red!10!blue!20!gray!30!white]
\tikzstyle{shaded line}=[double=red!10!blue!20!gray!30!white, double distance=1.5mm, draw=black]
\tikzstyle{unshaded}=[fill=white]
\tikzstyle{unshaded line}=[double=white, double distance=1.5mm, draw=black]
\tikzstyle{Tbox}=[circle, draw, thick, fill=white, opaque,]
\tikzstyle{empty box}=[circle, draw, thick, fill=white, opaque, inner sep=2mm]
\tikzstyle{background rectangle}= [fill=red!10!blue!20!gray!40!white,rounded corners=2mm] 
\tikzstyle{on}=[very thick, red!50!blue!50!black]
\tikzstyle{off}=[gray]

\tikzstyle{traces}=[scale=.2, inner sep=1mm]
\tikzstyle{quadratic}=[scale=.4, inner sep=1mm, baseline]
\tikzstyle{annular}=[scale=.7, inner sep=1mm, baseline]
\tikzstyle{make triple edge size}= [scale=.4, inner sep=1mm,baseline] 
\tikzstyle{icosahedron network}=[scale=.3, inner sep=1mm, baseline]
\tikzstyle{ATLsix}=[scale=.25, baseline]
\tikzstyle{TL12}=[scale=.15,baseline]
\tikzstyle{PAdefn}=[scale=.7,baseline]
\tikzstyle{TLEG}=[scale=.5,baseline]

\makeatletter

\newtheorem{lemma}{Lemma}[section]
\newtheorem{definition}[lemma]{Definition}
\newtheorem{definition*}{Definition}
\newtheorem{theorem}[lemma]{Theorem}
\newtheorem{proposition}[lemma]{Proposition}
\newtheorem{remark}[lemma]{Remark}
\newtheorem{corollary}[lemma]{Corollary}
\newtheorem{conjecture}[lemma]{Conjecture}
\newtheorem{example}[lemma]{Example}

\newtheorem{question}[lemma]{Question}

\makeatother

 \title{On Boolean intervals of finite groups}
\author{Mamta Balodi and Sebastien Palcoux} 
\address{Department of Mathematics, Indian Institute of Science, Bangalore, India}
\email{mamta.balodi@gmail.com} 
\address{Institute of Mathematical Sciences, Chennai, India}
\email{sebastien.palcoux@gmail.com}

\begin{document}

\begin{abstract}
We prove a dual version of {\O}ystein Ore's theorem on distributive intervals in the subgroup lattice of finite groups, having a nonzero dual Euler totient $\hat{\varphi}$. For any Boolean group-complemented interval, we observe that $\hat{\varphi} = \varphi \neq 0$ by the original Ore's theorem. We also discuss some applications in representation theory. We conjecture that $\hat{\varphi}$ is always nonzero for Boolean intervals. In order to investigate it, we prove that for any Boolean group-complemented interval $[H,G]$, the graded coset poset $\hat{P} = \hat{C}(H,G)$ is Cohen-Macaulay and the nontrivial reduced Betti number of the order complex $\Delta(P)$ is $\hat{\varphi}$, so nonzero. We deduce that these results are true beyond the group-complemented case with $|G:H|<32$. One observes that they are also true when $H$ is a Borel subgroup of $G$.
\end{abstract}

\maketitle

\section{Introduction}  
An extension of {\O}ystein Ore's result \cite[Theorem 7]{or} into the framework of planar algebras was investigated by the second author. It led to \cite[Conjecture 1.6]{pa} which admits two group-theoretical translations dual to each other \cite[Theorem 6.11]{pa}. One of them recovers the original theorem and the other is the dual version which we investigate here.   

Throughout the paper, an \emph{interval of finite groups} $[H,G]$ will always mean an interval in the subgroup lattice of the finite group $G$, with $H$ as a subgroup. 

Section \ref{2} consists of some basics (which are freely used in this introduction) about lattices, order complex, Cohen-Macaulay posets, edge labeling and GAP coding. In Section \ref{seclinprim}, we first prove a generalization of the following Ore's theorem to any top Boolean interval.

\begin{theorem} \label{oreintro}
 Let $[H,G]$ be a distributive interval of finite groups. Then there exists $g \in G$ such that $\langle H,g \rangle = G$.
\end{theorem}

Then we investigate a dual version.

\begin{definition} \label{euler}
Let $[H,G]$ be an interval of finite groups. Its \emph{Euler totient} $\varphi(H,G)$ is the number of cosets $Hg$ such that $\langle Hg \rangle = G$. \emph{Note that $\langle Hg \rangle = \langle H,g \rangle$.}
\end{definition}

Similar to Hall's argument in \cite{hal}, for any $K \in [H,G]$, $\sum_{L \in [H,K]} \varphi(H,L)$ is precisely $|K:H|$, so by M\"obius inversion formula,

$$\varphi(H,G)=\sum_{K \in [H,G]} \mu(K,G) |K:H|.$$

\begin{definition} \label{dualeuler}
Let $[H,G]$ be an interval of finite groups. Its \emph{dual Euler totient} is $$\hat{\varphi}(H,G):=\sum_{K \in [H,G]} \mu(H,K) |G:K|.$$
\end{definition}

%\begin{remark} \label{topbottomeuler}
Let $[T,G]$ be the top interval of $[H,G]$. By the crosscut theorem \cite[Corollary 3.9.4]{sta}, $\mu(L,G) = 0$ for all $L \in [H,G] \setminus [T,G]$, so  $$\varphi(H,G) = |T:H| \cdot \varphi(T,G).$$ Then, for $G=C_n$, the cyclic group of order $n$, $H=1$ and $n = \prod_i p_i^{n_i}$ the prime factorization, we recover the usual formula for the Euler's totient function $\varphi(n)$: $$\varphi(1,C_n)= \prod_ip_i^{n_i-1} \cdot  \prod_i(p_i-1).$$ Now, by applying the crosscut theorem on the dual of $[H,G]$, we deduce that if $[H,B]$ is the bottom interval of $[H,G]$, then $\mu(H,L) = 0$ for all $L \in [H,G] \setminus [H,B]$, so $$\hat{\varphi}(H,G)= |G:B| \cdot \hat{\varphi}(H,B).$$
%\end{remark}

 Next, we prove the following result which is a dual version of Ore's theorem.

\begin{theorem} \label{alterindexintro}
 Let $[H,G]$ be a distributive interval of finite groups. If its dual Euler totient $\hat{\varphi}(H,G)$ is nonzero, then there is an irreducible complex representation $V$ of $G$, such that $G_{(V^H)} = H$ (see Definition \ref{fixstab}).
\end{theorem}

As with Theorem \ref{oreintro}, Theorem \ref{alterindexintro} also extends to any bottom Boolean interval. We deduce several applications of the above theorem as a criterion for a finite group to have an irreducible faithful complex representation. We also find a non-trivial upper bound for the minimal number of irreducible components for a faithful complex representation of any finite group.    

We observe that for any Boolean interval $[H,G]$ satisfying the group-complemented assumption, i.e. $KK^{\complement} = G$ for all $ K \in [H,G]$, the dual Euler totient $\hat{\varphi}(H,G)$ agrees with the Euler totient $\varphi(H,G)$, hence is nonzero by Ore's theorem.  This leads us to make the following conjecture which implies that the condition on $\hat{\varphi}$ in Theorem \ref{alterindexintro} is vacuous; this would prove the conjectures in \cite[Section 6B]{pa}.

\begin{conjecture} \label{hatchi} If  $[H,G]$ is  Boolean, then $\hat{\varphi}(H,G) \neq 0$. 
\end{conjecture} 

Section \ref{CM} exposes the first results we obtained by investigating Conjecture \ref{hatchi}. For any interval $[H,G]$, the M\"obius invariant of its bounded coset poset is
$$\mu(\hat{C}(H,G)) =  -\sum_{K \in [H,G]} \mu(K,G) |G:K|,$$
and we observe (after Woodroofe) that in the rank $n+1$ Boolean case, $\mu(\hat{C}(H,G))$ is exactly $(-1)^{n}\hat{\varphi}(H,G)$.
Thus Conjecture \ref{hatchi} reduces to investigate the non-vanishing of the M\"obius invariant $\mu(\hat{C}(H,G))$, for any Boolean interval $[H,G]$. As explained in \cite[p759]{shawo}, it is unknown whether $\mu(\hat{C}(1,G))$ is nonzero in general, but Gasch\"utz proved it for any solvable group. A weaker version, namely, $\Delta(C(1,G))$ is not contractible, asked by Brown in \cite[Question 4]{br}, was proved by Shareshian and Woodroofe in \cite{shawo}. This leads us to the following weaker version of Conjecture \ref{hatchi}, which is also investigated for some cases. 

\begin{conjecture}
 If  $[H,G]$ is  Boolean, then $\Delta(C(H,G))$ is not contractible.  
\end{conjecture}  

It is well-known that the M\"obius invariant of a poset $P$ is the (reduced) Euler characteristic of its order complex $\Delta(P)$, which, by Euler-Poincar\'e formula, is  
$$\tilde{\chi}(\Delta(P)) =  \sum_{k=-1}^{n} (-1)^k \tilde{\beta}_k (\Delta(P)), $$ with $\dim(\Delta(P))=n$ and  
 $\tilde{\beta}_k (\Delta(P))$ the $k$th reduced Betti number, i.e. the dimension of $k$th reduced homology space. By definition, if $\hat{P}$ is Cohen-Macaulay, then $\tilde{\beta}_k (\Delta(P)) = 0$ for all $k < n$, so that $$\tilde{\beta}_n (\Delta(P)) = (-1)^n \mu(\hat{P}),$$ which is equal to $\hat{\varphi}(H,G)$ for $\hat{P} = \hat{C}(H,G)$ when $[H,G]$ is Boolean. Therefore, we look for the condition under which $\hat{C}(H,G)$ is Cohen-Macaulay.
 
  A sufficient condition for a graded poset to be Cohen-Macaulay is the existence of a (dual) EL-labeling \cite{bgs}. Woodroofe suggested us a labeling for $\hat{C}(H,G)$ in the case where $[H,G]$ is Boolean. We prove that it is a dual EL-labeling if and only if $[H,G]$ is, in addition, group-complemented, which leads to:
\begin{theorem} 
 Let $[H,G]$ be a Boolean group-complemented interval of rank $n+1$. Then $\hat{C}(H,G)$ is Cohen-Macaulay and $ \tilde{\beta}_n (\Delta(C(H,G))) = \varphi(H,G)$ is nonzero.
 % is the Euler totient $\varphi(H,G)$, and in particular, is nonzero.
\end{theorem} 

Note that at index $|G :H|<32$, there are $612$ Boolean intervals up to equivalence (see Subsection \ref{gap}). They are all group-complemented, except $[D_8,A_2(2)]$ and $[S_3,A_2(2)]$, both of rank $2$. The rank $2$ case is known to be Cohen-Macaulay and we observed that $\hat{\varphi}$ is also nonzero. This leads us to ask:

 \begin{question} \label{QCM} Is $\hat{C}(H,G)$ Cohen-Macaulay for any Boolean interval $[H,G]$, and the nontrivial reduced Betti number of $\Delta(C(H,G))$ nonzero? \end{question}

If $G$ has a BN-pair (as for any finite simple group of Lie type) of rank $n$ with $H$ being the corresponding Borel subgroup, then $[H,G]$ is Boolean of rank $n$. The order complex of $C(H,G)$ is equivalent to the building associated to the BN-pair \cite{bjo}. This leads to a positive answer to Question \ref{QCM} in this case.

\tableofcontents

\vspace*{.5cm}

\section{Preliminaries} \label{2}
\subsection{Lattices basics} \label{basics} 
We refer to \cite{sta} for the definition of \emph{finite lattice} $L$, \emph{meet} $\wedge$, \emph{join} $\vee$, \emph{subgroup lattice} $\mathcal{L}(G)$, \emph{sublattice} $L' \subseteq L$, \emph{interval} $ [a,b] \subseteq L$, \emph{minimum} $\hat{0}$, \emph{maximum} $\hat{1} $, \emph{atom}, \emph{coatom}, \emph{distributive lattice}, \emph{Boolean lattice} $\mathcal{B}_n$ (of rank $n$), \emph{complement} $b^{\complement}$ (with $b \in \mathcal{B}_n$). The \emph{top interval} of a finite lattice is the interval $ [t,\hat{1}] $, with $ t $ the meet of all the coatoms; and the \emph{bottom interval} is $ [\hat{0},b] $, with $ b $ the join of all the atoms. A lattice with a Boolean top (resp. bottom) interval will be called \emph{top} (resp. \emph{bottom})  \emph{Boolean}. Birkhoff representation theorem states that a finite lattice is distributive if and only if it embeds into some $ \mathcal{B}_n $, then \cite[Section 3.4, items a-i, p254-255]{sta} leads us to:

\begin{lemma} \label{topBn}
A finite distributive lattice is top and bottom Boolean.
\end{lemma}  

\subsection{Order complex} \label{CoMa}
All posets in this paper are assumed to be finite. A poset $P$ is bounded if it admits a minimum element $\hat{0}$ and a maximum element $\hat{1}$.
Given a poset $P$, its bounded extension is defined as $\hat{P}:=P \sqcup \{\hat{0},\hat{1} \}$. The proper part of $Q := \hat{P}$ is given by $\bar{Q}:=P$. One can associate to $P$ an abstract simplicial complex $\Delta(P)$, called its \emph{order complex}, defined as follows: the vertices of $\Delta(P)$ are the elements of $P$ and the faces of $\Delta(P)$ are the chains (i.e., totally ordered subsets) of $P$. Any topological property attributed to $\Delta(P)$ will be considered as for its geometric realization. We refer to \cite[Section 1.1]{wa} for more details.  The \emph{reduced Euler characteristic} $\tilde{\chi}(\Delta)$ of a simplicial complex $\Delta$ is
 $$\tilde{\chi}(\Delta) := \sum_{\substack{i=-1}}^{\substack{\dim(\Delta)}}~ (-1)^i f_i(\Delta),$$
where $f_i(\Delta)$ is the number of $i$-faces of $\Delta$.

\begin{example}[\cite{wa}, p7]  The order complex of $\bar{\mathcal{B}}_n$ is the barycentric subdivision of the boundary of the $n$-simplex, it has the homotopy type of $\mathbb{S}^{n-2}$. 
\end{example}
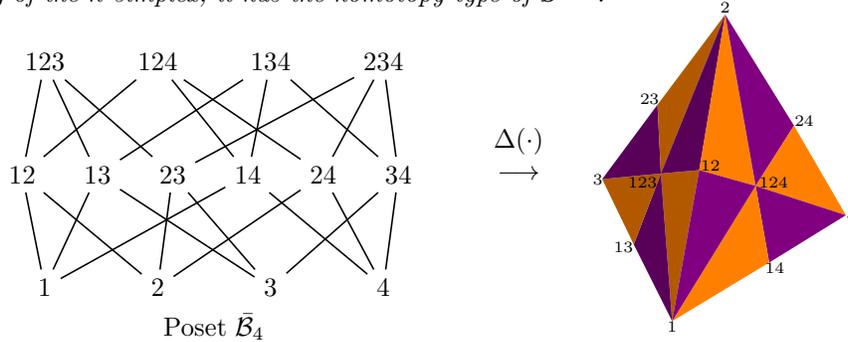
\vspace*{-7mm}
$$\begin{tikzpicture}[scale=1.5]
\node (A123) at (-2,1) {$123$};
\node (A124) at (-1,1) {$124$};
\node (A134) at (0,1) {$134$};
\node (A234) at (1,1) {$234$};
\node (A12) at (-3*2/3-0.2,0) {$12$};
\node (A13) at (-2*2/3-0.2,0) {$13$};
\node (A23) at (-1*2/3-0.2,0) {$23$};
\node (A14) at (0*2/3-0.2,0) {$14$};
\node (A24) at (1*2/3-0.2,0) {$24$};
\node (A34) at (2*2/3-0.2,0) {$34$};
\node (A1) at (-2,-1) {$1$};
\node (A2) at (-1,-1) {$2$};
\node (A3) at (0,-1) {$3$};
\node (A4) at (1,-1) {$4$};
\node  at (-0.5,-1.35) {Poset $\bar{\mathcal{B}}_4$};
\tikzstyle{segm}=[-,>=latex, semithick]
 \draw [segm] (A1)to(A12); \draw [segm] (A1)to(A13); \draw [segm] (A1)to(A14);
 \draw [segm] (A2)to(A12);\draw [segm] (A2)to(A23);  \draw [segm] (A2)to(A24); 
 \draw [segm] (A3)to(A13); \draw [segm] (A3)to(A23); \draw [segm] (A3)to(A34); 
 \draw [segm] (A4)to(A14); \draw [segm] (A4)to(A24); \draw [segm] (A4)to(A34); 
 \draw [segm] (A123)to(A12); \draw [segm] (A123)to(A13); \draw [segm] (A123)to(A23);
 \draw [segm] (A124)to(A12);\draw [segm] (A124)to(A24);  \draw [segm] (A124)to(A14); 
 \draw [segm] (A134)to(A13); \draw [segm] (A134)to(A34); \draw [segm] (A134)to(A14); 
 \draw [segm] (A234)to(A24); \draw [segm] (A234)to(A23); \draw [segm] (A234)to(A34); 
 \hspace{0.3cm}
\node  at (2,0) {$\longrightarrow$};
\node  at (2,0.3) {$\Delta(\cdot)$}; 

\tiny 

\fill[color=black!30!orange!100] (3.56,-1.3) -- (3.8,0.04) -- (3.46,0.01) -- cycle;
\fill[color=black!30!violet!100] (3.56,-1.3) -- (3.22,-0.62)-- (3.46,0.01) -- cycle;
\fill[color=black!30!orange!100] (2.94,-0.04) -- (3.22,-0.62)-- (3.46,0.01) -- cycle;
\fill[color=black!30!violet!100] (2.94,-0.04) -- (3.43,0.62)-- (3.46,0.01) -- cycle;
\fill[color=black!30!orange!100] (4.03,1.42) -- (3.43,0.62)-- (3.46,0.01) -- cycle;
\fill[color=black!30!violet!100] (4.03,1.42) -- (3.8,0.04) -- (3.46,0.01) -- cycle;

\fill[color=orange] (4.03,1.42) -- (3.8,0.04) -- (4.3,-0.1) -- cycle;
\fill[color=violet] (4.03,1.42) -- (4.64,0.435) -- (4.3,-0.1) -- cycle;
\fill[color=orange] (5.1,-0.36) -- (4.64,0.435) -- (4.3,-0.1) -- cycle;
\fill[color=violet] (5.1,-0.36) -- (4.42,-0.78) -- (4.3,-0.1) -- cycle;
\fill[color=orange] (3.56,-1.3) -- (4.42,-0.78) -- (4.3,-0.1) -- cycle;
\fill[color=violet] (3.56,-1.3) -- (3.8,0.04) -- (4.3,-0.1) -- cycle;

\node (A123) at (3.3-0.2,-0.07) {$123$};
\node (A124) at (4.46-0.2,-0.07) {$124$};
\node (A12) at (3.9-0.2,0.08) {$12$};
\node (A13) at (3.13-0.2,-0.64) {$13$};
\node (A23) at (3.36-0.2,0.67) {$23$};
\node (A14) at (4.47-0.2,-0.83) {$14$};
\node (A24) at (4.73-0.2,0.48) {$24$};
\node (A1) at (3.56-0.2,-1.35) {$1$};
\node (A2) at (4.03-0.2,1.48) {$2$};
\node (A3) at (2.9-0.2,-0.04) {$3$};
\node (A4) at (5.15-0.2,-0.36) {$4$};
\end{tikzpicture}
$$ 

\noindent We have $\tilde{\chi}(\Delta(\bar{\mathcal{B}}_4))=-1+14-36+24=1.$ \\

\noindent For a simplicial complex $\Delta$, there is the Euler-Poincar\'e formula
$$\tilde{\chi}(\Delta)=\sum_{i=-1}^{\text{dim} (\Delta)}~ (-1)^i \tilde{\beta}_i(\Delta),$$
 where $\tilde{\beta}_i(\Delta)$ is the reduced Betti number (i.e., dimension of $i$th reduced homology) of $\Delta$. For posets (co)homology, we refer to the survey \cite[Section 1.5]{wa}.
 
 We refer to \cite{sta} for the notion of M\"obius function $\mu$ on a poset, and its properties. We denote the M\"obius invariant of a bounded poset $P$ by $\mu(P) := \mu(\hat{0},\hat{1})$. For any poset $P$, $\mu(\hat{P})=\tilde{\chi}(\Delta(P))$ \cite[Proposition 3.8.6]{sta}. For two bounded posets $P_1$ and $P_2$, we have $\mu(P_1 \times P_2)=\mu_1(P_1) \times \mu_2(P_2)$. Therefore $\mu(\mathcal{B}_n) = (-1)^n$, because $\mu(\mathcal{B}_1) = -1$. 
%$\mathcal{B}_{n}=\mathcal{B}_{n-1} \times\mathcal{B}_1$ and

\subsection{Cohen-Macaulay posets and edge labeling}
A poset $P$ is called \emph{pure} if all the maximal chains have the same length $\ell(P)$; and is called \emph{graded} if it is both pure and bounded. A cover relation $x \lessdot y$ in a poset $P$ is the relation $x<y$ such that $x \le z<y$ implies $x=z$. The cover relations on $P$ will be identified with the edges on its Hasse diagram.  An edge labeling of a poset $P$ is a map $\lambda: E(P) \to A$, where $E(P)$ is the set of edges and $A$ is some poset (for our purpose $A$ will be the integers).

\begin{definition}[\cite{wa}] Let $\lambda$ be an edge labeling of a graded poset $P$. $\lambda$ is said to be an \emph{edge-lexicographical labeling} (or \emph{EL-labeling}) if for each closed interval $[x, y]$
of $P$, there is a unique strictly increasing  maximal chain, and this maximal chain is lexicographically first among all other maximal chains of $[x, y]$. A \emph{dual EL-labeling} on a graded poset $P$ is an EL-labeling on its dual poset (i.e. order reversed).
\end{definition}

\begin{example} \label{B3label} Here is an EL-labeling for $\mathcal{B}_3$ (it generalizes to $\mathcal{B}_n$).
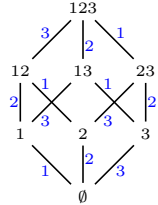
\begin{center} $\begin{tikzpicture}[scale = 1.5] \tiny
\node (A123) at (0,0) {$123$};
\node (A12) at (-1/1.8,-1/1.8) {$12$};
\node (A13) at (0,-1/1.8) {$13$};
\node (A23) at (1/1.8,-1/1.8) {$23$};
\node (A1) at (-1/1.8,-2/1.8) {$1$};
\node (A2) at (0,-2/1.8) {$2$};
\node (A3) at (1/1.8,-2/1.8) {$3$};
\node (A0) at (0,-3/1.8) {$\emptyset$};
\node  at (0.6/1.8,-2.6/1.8) {\color{blue}{$3$}};
\node  at (0.1/1.8,-2.4/1.8) {\color{blue}{$2$}};
\node  at (-0.6/1.8,-2.6/1.8){\color{blue}{$1$}};
\node  at (-1.1/1.8,-1.5/1.8){\color{blue}{$2$}};
\node  at (-0.6/1.8,-1.8/1.8){\color{blue}{$3$}};
\node  at (0.4/1.8,-1.8/1.8){\color{blue}{$3$}};
\node  at (1.1/1.8,-1.5/1.8){\color{blue}{$2$}};
\node  at (0.4/1.8,-1.2/1.8){\color{blue}{$1$}};
\node  at (-0.6/1.8,-1.2/1.8){\color{blue}{$1$}};
\node  at (0.6/1.8,-0.4/1.8) {\color{blue}{$1$}};
\node  at (0.1/1.8,-0.6/1.8) {\color{blue}{$2$}};
\node  at (-0.6/1.8,-0.4/1.8){\color{blue}{$3$}};
\tikzstyle{segm}=[-,>=latex, semithick]
\draw[<->] [segm] (A123)to(A12); \draw [segm] (A13)to(A1); 
\draw [segm] (A123)to(A13); \draw [segm] (A13)to(A3);  
\draw [segm] (A123)to(A23); \draw [segm] (A23)to(A2);  
\draw [segm] (A0)to(A2); \draw [segm] (A0)to(A3); 
 \draw [segm] (A0)to(A1); \draw [segm] (A12)to(A2);
\draw [segm] (A12)to(A1); \draw [segm] (A23)to(A3); 

\end{tikzpicture} 
$ \end{center}
\end{example}

A graded poset $P$ is called \emph{Cohen-Macaulay} \cite{bgs} (over $\mathbb{C}$) if for any open interval $(x,y)$ in $P$, $\tilde{\beta}_i(\Delta((x,y)),\mathbb{C})=0$ for all $i<$ $\dim(\Delta((x,y)))$.

\begin{theorem}\label{ELCM}
If  $\hat{P}$ is a graded poset which admits an EL-labeling, then it is Cohen-Macaulay. Moreover, the order complex $\Delta(P)$ has the homotopy type of a wedge of spheres $\mathbb{S}^d$ with $d=\ell(P)$. The number of spheres is one of the following equal quantities:
\begin{enumerate}
\item the number of (weakly) decreasing maximal  chains in $\hat{P}$;
\item the M\"obius invariant $\mu(\hat{P})$ times $(-1)^d$;
\item the reduced Betti number $\tilde{\beta}_d(\Delta(P))$;
\item the reduced Euler characteristic $\tilde{\chi}(\Delta(P))$ times $(-1)^d$.
\end{enumerate}
\end{theorem}
\begin{proof}
Merge \cite[Theorem 3.2]{bgs} and \cite[Theorem 3.2.4]{wa}.
\end{proof}

Note that a Cohen-Macaulay poset need not have an EL-labeling (see \cite[p16]{bgs}).

\begin{remark} \label{dualrem} The order complex is invariant by dual, so is Cohen-Macaulayness. A dual EL-labeling has essentially the same consequences as an EL-labeling.
\end{remark}

The EL-labeling of Example \ref{B3label} on $\mathcal{B}_n$ makes it Cohen-Macaulay. Moreover, $\mu(\mathcal{B}_n) = (-1)^n$, so Theorem \ref{ELCM} shows why $\Delta(\bar{\mathcal{B}}_n)$ has the homotopy type of $\mathbb{S}^{n-2}$.

\subsection{GAP coding} \label{gap} 
Two intervals of finite groups $[H_i,G_i]$, $i=1,2$, are equivalent if there is an isomorphism $\phi: G_1/K_1 \to G_2/K_2$ such that $\phi(H_1/K_1)=H_2/K_2$, with $K_i$ the normal core of $H_i$ in $G_i$. For any interval of finite groups $[H,G]$, $G$ acts transitively by permutation on the finite set $G/H$ of cardinal $|G:H|$ called the degree of the action. On the other hand, for any transitive action of a finite group $G$ on the finite set $[n]:=\{1,\dots, n\}$, the stabilizer subgroup $H=G_{\{1\}}$ is of index $n$. Up to equivalence, the data of an interval of finite groups of index $n$ is given by a transitive permutation group of degree $n$, i.e. a subgroup of $S_n$ acting transitively on $[n]$. They are completely classified (up to equivalence) on GAP \cite{gap} up to degree $31$.

\section{Ore's theorem and dual version} \label{seclinprim}
\subsection{Ore's theorem on Boolean intervals of finite groups} \hspace*{1cm}\\
Ore has proved the following result in \cite[Theorem 4, p267]{or}.

\begin{theorem} \label{ore1}
 
A finite group $G$ is cyclic if and only if its subgroup lattice $\mathcal{L}(G)$ is distributive. 
\end{theorem} 
%\begin{proof}
% ($\Leftarrow$): It is just a particular case of Theorem \ref{ore2} with $H = \{1\}$.   \\
% ($\Rightarrow$): A finite cyclic group $G = C_n$ has exactly one subgroup of order $d$, denoted by $C_d$, for every divisor $d$ of $|G|$. Now $C_{d_1} \vee C_{d_2} = C_{lcm(d_1,d_2)}$ and $C_{d_1} \wedge C_{d_2} = C_{gcd(d_1,d_2)}$, but lcm and gcd are distributive, so the result follows.
%\end{proof}

Ore extended one side of this result to any distributive interval of finite groups in \cite[Theorem 7]{or}, and below we extend it to any top Boolean interval.

 \begin{definition} 
An interval of finite groups $[H,G]$ is called \emph{w-cyclic} (or \emph{weakly cyclic}) if there is $g \in G$ such that $\langle H,g \rangle = G$.
\end{definition}

\begin{lemma} \label{topred} 
An interval $[H , G]$ is w-cyclic if its top interval is so.
\end{lemma}
\begin{proof}
Let $[T,G]$ be the top interval and $g \in G$ with $\langle T,g \rangle = G$. For any maximal subgroup $M$ of $G$, we have $T \subseteq M$ by definition, and so $g \not \in M$,  then a fortiori $\langle H,g \rangle \not \subseteq M$. It follows that $\langle H,g \rangle=G$.
\end{proof}

\begin{theorem} \label{ore2}
If the interval $[H,G]$ is top Boolean, then it is w-cyclic.
\end{theorem} 
\begin{proof}
By Lemmas \ref{topBn} and \ref{topred}, the proof reduces to the Boolean case. We prove it by induction on the rank of the lattice. For rank $1$, consider $g \in G$ such that $g \not \in H$, then $\langle H,g \rangle = G$ by maximality. For rank $n>1$, let $M$ be a coatom of $[H,G]$, and $M^{\complement}$ be its complement. Using induction, we can assume $[H,M]$ and $[H,M^{\complement}]$  both to be w-cyclic, i.e. there are $a, b \in G$ such that $\langle H,a \rangle = M$ and $\langle H,b \rangle = M^{\complement}$. For $g=a  b$ we have $a=g  b^{-1}$ and $b=a^{-1}g$, so $$\langle H,a,g \rangle = \langle H,g,b \rangle = \langle H,a,b \rangle =  M \vee M^{\complement} = G.$$ Now, $\langle H,g \rangle =  \langle H,g \rangle \vee H = \langle H,g \rangle \vee (M \wedge M^{\complement})$ but by distributivity $$\langle H,g \rangle \vee (M \wedge M^{\complement}) = (\langle H,g \rangle \vee M \rangle) \wedge (\langle H,g \rangle \vee M^{\complement} \rangle).$$ So $ \langle H,g \rangle = \langle H,a,g \rangle \wedge \langle H,g,b \rangle = G$. The result follows.  
\end{proof}
Theorem  \ref{ore1} follows from Theorem \ref{ore2} and the fact that gcd and lcm are distributive over each other on $C_n$.
\subsection{A dual version of Ore's theorem} \hspace*{1cm} \\
In this section, we extend a result on intervals of finite groups $[H,G]$, obtained in a previous paper \cite[Section 6B]{pa}, of the second author. 

 \begin{definition} \label{fixstab} Let $W$ be a representation of a group $G$, $K$ a subgroup of $G$, and $X$ a subspace of $W$. We define the \emph{fixed-point subspace} $$W^{K}:=\{w \in W \ \vert \  kw=w \ , \forall k \in K  \},$$ and the \emph{pointwise stabilizer subgroup} $$G_{(X)}:=\{ g \in G \  \vert \ gx=x \ , \forall x \in X \}.$$  \end{definition} 
 
 A group $G$ having a faithful irreducible complex representation $V$ is sometimes called as \emph{linearly primitive} (note that faithfulness is equivalent to have $G_{(V)} = 1$). This notion extends to intervals as follows.

 \begin{definition} 
An interval of finite groups $[H,G]$ is called \emph{linearly primitive} if there is an irreducible complex representation $V$ of $G$ such that $G_{(V^H)} = H$.
\end{definition}
 
 \begin{lemma} \label{tech} Let $G$ be a finite group, $H,K$ two subgroups, $V$ a representation of $G$ and $X,Y$ two subspaces of $V$. Then
\begin{itemize}
\item[(1)] $H \subseteq K \Rightarrow V^{K} \subseteq V^{H}$,
\item[(2)] $X \subseteq Y \Rightarrow G_{(Y)} \subseteq G_{(X)}$,
\item[(3)] $V^{H \vee K} = V^H \cap V^K $,
\item[(4)] $H \subseteq G_{(V^H)}$,
\item[(5)] $ V^{G_{(V^H)}} = V^H$,
\item[(6)] $H \subseteq K$ and $V^{K} \subsetneq V^{H}$ imply $K \not \subseteq G_{(V^H)}$.
\end{itemize}
\end{lemma}

\begin{proof}
Straightforward.
% (1) and (2) are immediate. 
%\noindent (3) First $H,K \subseteq H \vee K$, so $V^{H \vee K} $ is included in $ V^H$ and  $V^{K}$, so in $V^H \cap V^{K}$. Now take $v \in V^H \cap V^{K}$, then $\forall h \in H$ and $\forall k \in K$, $hv = kv = v$, but any element $g \in H \vee K$ is of the form $h_1k_1h_2k_2 \cdots h_rk_r $ with $h_i \in H$ and $k_i \in K$, it follows that $gv=v$ and so $V^H \cap V^{K} \subseteq V^{H \vee K}$. (4) Take $h \in H$ and $v \in V^H$. Then by definition $hv = v$, so $H \subseteq  G_{(V^H)}$.  (5) From (1) and (4) we deduce that $V^{G_{(V^H)}} \subseteq V^H$. Now take $v \in V^H$ and $g \in G_{(V^H)}$, by definition $gv = v$, so $V^H \subseteq  V^{G_{(V^H)}}$ also. \\
%(6) Suppose that $K \subseteq G_{(V^H)}$, then $V^K \supseteq V^{G_{(V^H)}} = V^H$ by (1) and (5). Hence $V^K = V^H$ by (1), contradiction with $V^K \subsetneq V^H$.
\end{proof}

%
%\begin{remark} For $H = 1$, we recover the usual linear primitivity, i.e. the existence of an irreducible faithful complex representation. 
%\end{remark}

\begin{lemma} \label{indexrep}
Let $[H,G]$ be an interval of finite groups. Let $V_1, \dots , V_r$ be  the irreducible complex representations of $G$ (up to equivalence). Then $$|G:H| = \sum_i \dim(V_i)\dim(V_i^H).$$   
\end{lemma} 
\begin{proof}  Let $V = \mathbb{C}G$ be the regular representation of $G$.  By \cite{ser}, $ V \simeq \bigoplus_i V_i^{\oplus \dim(V_i)}. $ Then
 $$ V^H \simeq \bigoplus_i (V^H_i)^{\oplus \dim(V_i)}. $$
The result follows by comparing the dimensions on both sides, and observing that the space $V^H$ has a basis in bijection with $G/H$, so has dimension $|G:H|$.
\end{proof}  

\begin{lemma} \label{premain}
Let $[H,G]$ be a Boolean interval of finite groups, and $V$ be a finite dimensional complex representation of $G$. If $H \subsetneq G_{(V^H)}$, then $$\sum_{K \in [H,G]} \mu(H,K)\dim(V^K) = 0.$$ 
\end{lemma}
\begin{proof} Let $A$ be an atom of $[H,G]$ with $A \subseteq G_{(V^H)}$. By part (6) of Lemma \ref{tech}, $V^A = V^H$.  Consider the partition $[H,G] = [H,A^{\complement}] \sqcup [A,G],$ and the bijective map $$[H,A^{\complement}] \ni T \mapsto T \vee A \in [A,G].$$ Now $\mu(H,T \vee A) = -\mu(H,T)$, and by part (3) of Lemma \ref{tech}, $V^{T \vee A} = V^T \cap V^A = V^T$. So 
$$ \sum_{K \in [A,G]} \mu(H,K)\dim(V^K) = \sum_{T \in [H,A^{\complement}]} \mu(H,T\vee A)\dim(V^{T \vee A}), $$
$$\hspace*{6.22cm} = - \sum_{T \in [H,A^{\complement}]} \mu(H,T)\dim(V^T). \hspace*{2.32cm} \qedhere $$
\end{proof}

The following theorem is the main result of Section \ref{seclinprim}.

\begin{theorem} \label{main}
 Let $[H,G]$ be a Boolean interval of finite groups. If its dual Euler totient $\hat{\varphi}(H,G)$ is nonzero, then $[H,G]$ is linearly primitive.
\end{theorem}
\begin{proof}
If $[H,G]$ is not linearly primitive, then for any irreducible complex representation $V_i$ of $G$, we have that $H \subsetneq G_{(V_i^H)}$. So by Lemmas \ref{indexrep} and \ref{premain}
$$ \hat{\varphi}(H,G)= \sum_{K \in [H,G]} \mu(H,K) |G:K|  = \sum_{K \in [H,G]} \mu(H,K) \sum_i \dim(V_i)\dim(V_i^K)$$
$$ \hspace*{3.17cm} = \sum_i \dim(V_i) \sum_{K \in [H,G]}\mu(H,K)\dim(V_i^K) = 0. \hspace*{3.17cm} \qedhere $$
\end{proof}  

\begin{proposition} \label{chipos}
If the interval $[H,G]$ is Boolean, then $\varphi(H,G) > 0$.  
\end{proposition}
\begin{proof}
The number of cosets $Hg$ with $\langle Hg \rangle = G$ is nonzero by Theorem \ref{ore2}.
\end{proof} 

Let $[H,G]$ be a Boolean interval of finite groups. In general  $\varphi(H,G) \neq \hat{\varphi}(H,G) $. Here is the smallest example: let $A_2(2)$ be the simple group of order $168$, and $D_8$ a subgroup isomorphic to the dihedral group of order $8$. Then the interval $[D_8,A_2(2)]$ has $\varphi = 16$ and $\hat{\varphi} = 8$. 

\begin{lemma} \label{eupri} Let $[H,G]$ be an interval of finite groups with $|G:H|$ a prime-power $p^m$ and $p$ not a divisor of $\mu(H,G)$. Then $\hat{\varphi}(H,G)$ is nonzero.
\end{lemma}
\begin{proof}
 $\hat{\varphi}(H,G)= \sum_{K \in [H,G)} \mu(H,K) |G:K|  +  \mu(H,G)$ is nonzero because the first component of the sum is a multiple of $p$, whereas the second is not.
\end{proof}

This result can be applied in the Boolean case because $\mu(\mathcal{B}_n) = (-1)^n$. This leads us to Conjecture \ref{hatchi}. More strongly, we wonder whether $\varphi(H,G) $ and $\hat{\varphi}(H,G)$ are bounded below by $ 2^n$, in case $[H,G]$ is Boolean of rank $n+1$. We have checked it by GAP for $|G:H|<32$. If this lower bound is correct, then it is optimal because it is realized by $[1 \times S_2^n , S_2 \times S_3^n]$.  See also Example \ref{exBN}.

\begin{lemma} \label{grpcp}
Let $[H,G]$ be a Boolean interval of finite groups, let $K \in [H,G]$ and $K^{\complement}$ be its lattice-complement. The following are equivalent:
\begin{itemize}
\item[(1)] $KK^{\complement} = K^{\complement}K$,
\item[(2)] $KK^{\complement} = G$,
\item[(3)] $|G:K| = |K^{\complement}:H|$.
\end{itemize}
\end{lemma}
\begin{proof}
$(1) \Rightarrow (2)$: We have $K \vee K^{\complement} = G$ and $K \wedge K^{\complement} = H$. But any element in $K \vee K^{\complement}$ is of the form $k_1k'_1k_2k'_2 \cdots k_rk'_r $ with $k_i \in K$ and $k'_i \in K^{\complement}$. So by $(1)$ any such element is of the form $kk'$ with $k \in K$ and $k' \in K^{\complement}$, i.e. $G=KK^{\complement}$. \\ 
$(1) \Leftarrow (2)$: Immediate. \\ 
$(2) \Leftrightarrow (3)$:  By the product formula, $|KK^{\complement}| = |K||K^{\complement}|/|H|$, so $KK^{\complement} = G$ if and only if $|KK^{\complement}| = |G| $, if and only if $|G:K|= |G|/|K| = |K^{\complement}|/|H| = |K^{\complement}:H|$. 
\end{proof}

\begin{definition} \label{grpcpdef}
A Boolean interval $[H,G]$ is called \emph{group-complemented} if every $K \in [H,G]$ satisfies one of the equivalent statements of Lemma \ref{grpcp}.
\end{definition}

\begin{remark}
There are Boolean intervals of finite groups which are not group-complemented, the smallest example is $[D_8,A_2(2)]$ of index $21$ and rank $2$. It is not group-complemented, because by GAP there is $K \in [D_8,A_2(2)]$ with $$|D_8:K| = 7 \neq 3 = |K^{\complement}:A_2(2)|.$$   
For groups with $|G:H|<32$, there is only one other example, given by $[S_3,A_2(2)]$.
\end{remark}

\begin{lemma}  \label{grpcplem}
If $[H,G]$ is Boolean group-complemented, then $\hat{\varphi}(H,G) = \varphi(H,G)$. 
\end{lemma}
\begin{proof} 
We have 
$$\varphi(H,G)= \sum_{K \in [H,G]} \mu(K,G) |K:H|, $$
but $\mu(K,G) = \mu(H,K^{\complement})$; moreover, by the group-complemented assumption and Lemma \ref{grpcp}, $|K:H| = |G:K^{\complement}|$, so 
$$\varphi(H,G)= \sum_{K \in [H,G]} \mu(H,K^{\complement}) |G:K^{\complement}|. $$  The result follows by the change of variable $K \leftrightarrow K^{\complement}$. 
\end{proof} 

\begin{corollary}  \label{cplchi}
Let $[H,G]$ be a Boolean interval. If it is group-complemented, then $\hat{\varphi}(H,G) > 0$. 
\end{corollary}
\begin{proof} 
By Lemma \ref{grpcplem} and Proposition \ref{chipos}.
\end{proof}

\begin{remark} \label{cplconv}
The converse of Corollary \ref{cplchi} is false as $[D_8 , A_2(2)]$ is Boolean and not group-complemented with $\hat{\varphi} = 8 >0$.
\end{remark} 

\begin{corollary}  \label{corocpl}
If a Boolean interval of finite groups $[H,G]$ is group-complemented, then it is linearly primitive. 
\end{corollary}
\begin{proof} 
By Corollary \ref{cplchi}, $\hat{\varphi}(H,G) \neq 0$, so we apply Theorem \ref{main}.   
\end{proof} 

We say that an interval $[H,G]$ is \emph{Dedekind} if every $K \in [H,G]$ and every $g \in G$ satisfy  $HgK = KgH$. The case $H=1$ corresponds to the usual notion of Dedekind group.

\begin{lemma} \label{dedecomp} For a Boolean interval of finite groups,  Dedekind  implies group-complemented.
\end{lemma}
\begin{proof}
By assumption, for all $K \in [H,G]$ and for all $g \in G$, we have $HgK = KgH$. It follows that  $HK^{\complement}K = KK^{\complement}H$, but $HK^{\complement} = K^{\complement}H = K^{\complement}$, so $KK^{\complement} = K^{\complement}K$.
\end{proof}

\begin{remark} \label{compnotdede}
The converse of Lemma \ref{dedecomp} is not true. There are Boolean group-complemented intervals which are not Dedekind, for example $[H,G]$ as follows:
\begin{verbatim}
gap> G:=TransitiveGroup(d,r);  H:=Stabilizer(G,1);  
\end{verbatim} 
with $(d,r) = (10,4)$ or $(30,7)$.
\end{remark} 

By this remark, Corollary \ref{corocpl} is strictly stronger than a result in a previous paper (where \emph{group-complemented} is replaced by \emph{Dedekind}  \cite[Remark 6.17]{pa}). Moreover, by Remark \ref{cplconv}, Theorem \ref{main} is strictly stronger than Corollary \ref{corocpl}.

\subsection{Applications to representation theory} 
We get a criterion for a finite group to be linearly primitive. We also deduce a non-trivial upper bound for the minimal number of irreducible components for a faithful complex representation.

\begin{lemma} \label{bottomprim}
An interval $[H,G]$ is linearly primitive if its bottom interval is so.
\end{lemma}
\begin{proof}
Let $[H,K]$ be the bottom interval of $[H,G]$. By assumption, there is an irreducible complex representation $W$ of $K$ such that $K_{(W^H)} = H$. Let $V$ be an irreducible complex representation of $G$ such that its restriction on $K$ admits $W$ as subrepresentation. Now $W \subseteq V$, so that $W^H \subseteq V^H$ and hence
$$H \subseteq K_{(V^H)} \subseteq  K_{(W^H)} = H.$$ It follows that $K_{(V^H)}=H$.  If there is an atom $X \in [H,G]$ with $V^H = V^{X}$, then $$ X \subseteq  K_{(V^{X})}  = K_{V^H},$$ contradicting $H \subsetneq X$. So for every atom $X$, $V^{X} \subsetneq V^H$. Then by part (6) of Lemma \ref{tech}, $X \not \subseteq G_{(V^H)}$, so by minimality $G_{(V^H)} = H$.   
\end{proof}  
We can then provide a dual version of Theorem \ref{ore2}.
\begin{corollary}  \label{coromain}
A bottom Boolean interval $[H,G]$ with $\hat{\varphi}(H,G) \neq 0$ is linearly primitive. 
\end{corollary}
\begin{proof} Let $[H,B]$ be the bottom interval of $[H,G]$. Then $\hat{\varphi}(H,G) = |G:B|\hat{\varphi}(H,B)$. So the result follows by Lemma \ref{bottomprim} and  Theorem \ref{main}.
\end{proof} 
Let $G$ be a group. A proper subgroup $H$ is called \emph{core-free} if it does not contain any non-trivial normal subgroup of $G$. The following result is \emph{almost} a combinatorial criterion for a finite group to be linearly primitive.
\begin{theorem}  \label{linprimG}
A finite group $G$, with a core-free subgroup $H$ such that $[H,G]$ is bottom Boolean with a nonzero dual Euler totient, is linearly primitive.
\end{theorem}
\begin{proof} 
By Corollary \ref{coromain}, the interval $[H,G]$ is linearly primitive.  So there is an irreducible complex representation $V$ with $G_{(V^{H})} = H$. Now, $V^H \subseteq V$ so $G_{(V)} \subseteq G_{(V^H)}$, but $\ker(\pi_V) =  G_{(V)}$, it follows that $\ker(\pi_V) \subseteq H$; but $H$ is a core-free subgroup of $G$, and $\ker(\pi_V)$ a normal subgroup of $G$, so $\ker(\pi_V)= 1$, which means that $V$ is faithful on $G$, i.e. $G$ is linearly primitive.
\end{proof} 

As an easy consequence of Theorem \ref{linprimG}, we have:  if a finite group $G$ admits a core-free maximal subgroup, then it is linearly primitive.
%\begin{proof}
%Let $M$ be a maximal subgroup of $G$, then $[M,G]$ is Boolean (so equal to its bottom interval) of rank $1$ (so $\hat{\varphi}(M,G) \neq 0$). Hence, by Theorem \ref{linprimG}, if $M$ is core-free, then $G$ is linearly primitive.
%\end{proof}

\begin{question} \label{coco}
Is a finite group $G$ linearly primitive if and only if there is a core-free subgroup $H$ with $[H,G]$ bottom boolean?
\end{question}
We have checked it by GAP for any finite $p$-group $G$ of order less than $3^7$ (except $2^9, 2^{10}, 2^{11}$). Now, by \cite[Problem 5.25]{isa}, if every Sylow subgroup of a finite group $G$ is linearly primitive, then so is $G$. The following is a reformulation of Theorem \ref{linprimG} for $p$-groups.

\begin{corollary} \label{coqua} Let $G$ be a finite $p$-group with a core-free subgroup $H$ such that $N_G(H)/H$ is cyclic or generalized quaternion. Then $G$ is linearly primitive. 
\end{corollary}
\begin{proof}
A cyclic or generalized quaternion $p$-group has a unique subgroup of order $p$. So $[H,N_G(H)]$ has a unique atom $B'$. Since a maximal subgroup of a $p$-group is normal \cite[Corollary 1 p137]{zas}, $H$ is normal in any atom of $[H,G]$, so is in $B$ (the join of all the atoms). It follows that $B \subseteq N_G(H)$ the normalizer of $H$ in $G$, and $B=B'$. It follows that the bottom interval of $[H,G]$ is Boolean of rank $1$, thus $$\hat{\varphi}(H,G) = |G:B|\hat{\varphi}(H,B) = |G:B|(|B:H| - 1) \neq 0,$$ so by Theorem \ref{linprimG}, $G$ is linearly primitive.
\end{proof}

Note that any linearly primitive group has a cyclic center, and the two conditions are equivalent for $p$-groups \cite[Theorem 2.32]{isa}.   
%We consider the following converse to Theorem \ref{linprimG}.

%A positive answer to the above question leads to the following:
%\begin{statement} A finite $p$-group $G$ with a cyclic center has a core-free subgroup $H$ with $N_G(H)/H$ cyclic or generalized quaternion. 
%\end{statement}
%\begin{proof}
%Let $G$ be a finite $p$-group with a cyclic center. Then $G$ is linearly primitive, so there is a core-free subgroup $H$ such that the bottom interval $[H,B]$ of $[H,G]$ is Boolean. Then $B \subseteq N_G(H)$, as for Corollary \ref{coqua}. But $[H,B] \simeq [1,B/H]$ as lattices, so the subgroup lattice $\mathcal{L}(B/H)$ is Boolean, and by Ore's Theorem \ref{ore1}, $B/H$ is cyclic. Now it is also a $p$-group, so $B/H \simeq C_p^m$. But $\mathcal{L}(B/H)$ is Boolean, so $m=1$. It follows that $B$ is the unique atom of $[H,G]$, and $B/H$ is the unique atom of $\mathcal{L}(N_G(H)/H)$. So $N_G(H)/H$ is a $p$-group with a unique subgroup of order $p$, hence it is cyclic or generalized quaternion \cite[Theorem 15]{zas}.\end{proof}

%\begin{remark} By using the proof of Corollary \ref{coqua}, the above statement is in fact equivalent to a positive answer of Question \ref{coco} for any $p$-group $G$, which then is true if $|G|<512$, by GAP computation. Moreover, it is  true for $|G|<2187$, if $p>2$.
%\end{remark}

\begin{definition} \label{defbbelint} \label{bel}
Let $[H,G]$ be an interval of finite groups. Let $\lambda(H,G)$ be the minimal length for an ordered chain of subgroups $$H=H_0 \subseteq  H_1 \subseteq  \dots \subseteq  H_{n} = G$$ such that the interval $[H_{i-1} , H_{i}]$ is bottom Boolean and $\hat{\varphi}(H_{i-1} , H_{i})$ is nonzero. For $H=1$, we denote it just by $\lambda(G)$. It is a purely combinatorial invariant of $G$.
\end{definition}

\begin{theorem}  \label{bbelint}
For any interval $[H,G]$, the minimal number of irreducible complex representations  $V_1, \dots , V_n$ of $G$ such that $\bigcap_{i} G_{(V^H_{i})} = H$ is at most $\lambda(H,G)$.
\end{theorem}
\begin{proof} 
We use the same notation for a chain of length $\lambda(H,G)$ as in Definition \ref{defbbelint}. By Corollary \ref{coromain}, for any $i$ there exists an irreducible complex representation $W_{i}$ of $H_{i}$ such that ${H_{i}}_{({W_{i}}^{H_{i-1}})} = H_{i-1}$. Let $V_{i}$ be an irreducible complex representation of $G$ such that its restriction to $H_{i}$ admits $W_{i}$ as subrepresentation. Then, we standardly check that ${H_{i}}_{({W_{i}}^{H_{i-1}})} = G_{(V_{i}^{H_{i-1}})}$. Thus, for $\ell = \lambda(H,G)$,  
$$\bigcap_{i=1}^{\ell} G_{(V^H_{i})} \le  \bigcap_{i=1}^{\ell} G_{(V_{i}^{H_{i-1}})} = (\bigcap_{i=1}^{\ell-1} G_{(V_{i}^{H_{i-1}})}) \cap H_{\ell - 1} $$
$$ \hspace*{2.07cm} = \bigcap_{i=1}^{\ell-1} {H_{\ell - 1}}_{(V_{i}^{H_{i-1}})} = \cdots = \bigcap_{i=1}^{\ell-s} {H_{\ell - s}}_{(V_{i}^{H_{i-1}})} = \cdots = H_0 = H.  \hspace*{2.07cm} \qedhere $$
\end{proof}

\begin{corollary}  \label{leftreg}
For any finite group $G$, the minimal number of irreducible components for a faithful complex representation of $G$ is bounded above by $\lambda(G)$.
\end{corollary}
\begin{proof}
By Theorem \ref{bbelint}, we can find irreducible complex representations $V_1, \dots, V_{\lambda(G)}$ satisfying $\bigcap_{i=1}^{\lambda(G)} G_{(V_{i})} = 1$, but $G_{(V_{i})} = \ker(\pi_{i})$. So $$ \ker(\bigoplus_{i} \pi_{i}) = \bigcap_{i} \ker(\pi_{i}) = 1,$$ which implies that $V_1 \oplus \cdots \oplus V_{\lambda(G)}$ is faithful.
\end{proof}

Note that Theorem \ref{linprimG} allows us to improve the bound given in the previous corollary, but it requires knowledge of all normal subgroups.

\section{Cohen-Macaulay coset poset} \label{CM}

\subsection{M\"obius invariant of a coset poset} \label{moback}

For a group $G$ and its subgroup $H$, the \emph{coset poset} $C(H,G)$ is defined to be the poset of (proper) right cosets $Kg$, with $g \in G$ and $K \in [H,G)$, ordered by inclusion, and
 \begin{center}
 $\hat{C}(H,G):=C(H,G) \sqcup \{\emptyset,G \} $.
 \end{center}
 
\begin{lemma}\label{posetisom}  
The poset $[H_1,H_2]$ is isomorphic to the poset $[H_1g,H_2g]$, for $g \in G$ and $H_1, H_2 \in [H,G]$.  
\end{lemma}
\begin{proof}
The multiplication by $g$ is an order preserving bijection between the posets $[H_1,H_2]$ and $[H_1g,H_2g]$.
\end{proof} 

Let $P$ be a poset and $P^{\star}$ be the dual poset (order-reversed). Then $\mu(P^{\star}) = \mu(P)$. Therefore, for an interval of finite groups $[H,G]$, using Lemma \ref{posetisom}, we easily get the following relative version of Bouc's formula \cite[Section 3]{br},
$$ \mu(\hat{C}(H,G))= -\sum_{H \leq K \leq G} \mu(K,G) |G:K|.$$
\begin{lemma} \label{mobeul0}
If $[H,G]$ is Boolean of rank $n$, then $ \mu(\hat{C}(H,G)) =  -(-1)^n \hat{\varphi}(H,G). $ \end{lemma}
\begin{proof}
It follows from the fact that $\mu(K,G)=(-1)^n\mu(H,K)$, for $K \in [H,G]$. 
\end{proof}

If $\mu(\hat{C}(H,G))$ is nonzero, then $\Delta(C(H,G))$ is not contractible (see Subsection \ref{CoMa}). The remaining part of this subsection will be about non-contractibility. We first note that $G$ acts by right multiplication on $C(H,G)$, and then we have:

\begin{lemma}  \label{fpf}
Let $[H,G]$ be a w-cyclic interval of finite groups. Let $x \in G$ be such that $\langle H,x \rangle = G$. Then the cyclic group $\langle x \rangle$ acts fixed-point-freely by right multiplication on $C(H,G)$ if and only if for all $g \in G$, we have $\langle H,gxg^{-1} \rangle = G$.
\end{lemma}
\begin{proof}
Assume that there is $g \in G$ with $K:=\langle H,gxg^{-1} \rangle < G$. Then the proper coset $Kg$ is a fixed-point because $Kgx = Kg$ as $gxg^{-1} \in K$. Conversely, assume that $\langle H,gxg^{-1} \rangle=G$ for all $g \in G$. Suppose that $Kg$ is a coset fixed by $x$. Then the coset $Hg \vee Hgx$ is subset of $Kg$. But $Hg \vee Hgx = \langle H,gxg^{-1} \rangle g = G$ by assumption. The result follows.
\end{proof}

\begin{definition}
An interval of finite groups $[H,G]$ is called strongly w-cyclic if there is $x \in G$ such that $\langle H,gxg^{-1} \rangle = G$ for all $g \in G$, or equivalently, if $\bigcup_{g,i}g^{-1}M_ig \neq G$, with $M_1, \dots, M_n$ the coatoms of $[H,G]$.
\end{definition}

The Boolean group-complemented non-Dedekind intervals of Remark \ref{compnotdede} are not strongly w-cyclic, using GAP. Now, for $|G:H| < 32$ and $|G| \le 10^5$, there are (up to equivalence) $25608$ intervals, $21918$ of them are w-cyclic, $21773$ are top Boolean, $21708$ are strongly w-cyclic, $23323$ are Dedekind. Lastly, $21274$ are both w-cyclic and Dedekind, and all of them are also strongly w-cyclic. This leads us to ask whether any Dedekind w-cyclic interval $[H,G]$ is strongly w-cyclic. Note that if $H$ is a maximal subgroup, then $[H,G]$ is w-cyclic Dedekind, and also strongly w-cyclic, because the union of the conjugate subgroups of a proper subgroup (of a finite group) is always a proper subset, by the Orbit-Stabilizer Theorem. 

\begin{theorem}
Let $[H,G]$ be a strongly w-cyclic interval of finite groups. Then $\Delta(C(H,G))$ is not contractible.
\end{theorem}
\begin{proof}
By Lemma \ref{fpf}, there is a non-trivial cyclic group acting fixed-point-freely by automorphisms on $C(H,G)$, so the result follows by \cite[Theorem 3.4]{shawo}. 
\end{proof}

\noindent As for Lemma \ref{topred}, an interval is strongly w-cyclic if its top interval is so. Moreover:

\begin{theorem}
If $[T,G]$ is the top interval of $[H,G]$, then $C(H,G)$ is homotopy-equivalent to $C(T,G)$.
\end{theorem}
\begin{proof}
Consider the map $f:C(H,G) \ni Kg \mapsto T \vee K g \in C(T,G)$. For any $Lg \in C(T,G)$, observe that $f^{-1}(C(T,G)_{\le Lg})$ has a maximum element ($Lg$ itself). The result follows by Quillen fiber lemma \cite[Proposition 1.6]{qui}.  
\end{proof}

% $f^{-1}(C(T,G)_{\le Lg})$
%\begin{proposition}
%The poset $C(H,G)$ is homotopy-equivalent to $C(T,G)$, with $[T,G]$ the top interval of $[H,G]$.
%\end{proposition}
%\begin{proof}
%Consider the poset map $C(H,G) \ni Kg \mapsto T \vee K g \in C(T,G)$. For any coset $Lg$ of $C(T,G)$, the inverse image of $\{L'g' \in C(T,G) \ | \  L'g' \le L g \}$ has greatest element $L g$. Thus, such inverse images are contractible. The Quillen fiber lemma \cite[Theorem 10.5]{bjo2} yields the desired homotopy equivalence.
%\end{proof}
\subsection{An edge labeling for $\hat{C}(H,G)$}

Let us fix some notations which shall be used in the sequel. The interval of finite groups $[H,G]$ will be Boolean. Then the poset $\hat{C}(H,G)$ is graded. Let $K_1, \ldots, K_n$ be the atoms in the interval $[H,G]$. Then, the coatoms are of the form $M_i:={K_i}^{\complement} =\vee_{j \neq i} K_j$. We observe that:

\begin{lemma}
Any cover relation of $\hat{C}(H,G)$ is of the form $\emptyset \lessdot Hg$ or $Xg \lessdot Yg$ with $g \in G$, $Y=X \vee K_i$ for some $i$ and $K_i \nsubseteq X$.
\end{lemma}

The following edge labeling was suggested to us by Woodroofe on MathOverflow, and is also closely related to that of \cite{russ3, russ2}.
\begin{definition} \label{el} Let \emph{$el$} be the following edge labeling on $\hat{C}(H,G)$:
\begin{itemize}
\item $el(\emptyset \lessdot Hg) = 0$
\item $el(Xg \lessdot Yg)= \begin{cases}
 						-i~ \hspace{.2cm} \text{if}~ Xg=Yg \cap M_i,\\
 						+i~ \hspace{.2cm} \text{otherwise.}
 					\end{cases}$
\end{itemize}
\end{definition} 
\begin{lemma} \label{ginmi}
Let $Y=X \vee K_i$ with $K_i \nsubseteq X$. Then, $$Xg=Yg \cap M_i \Leftrightarrow g \in M_i.$$
\end{lemma}
\begin{proof}
%Suppose that $[H,G]$ is Boolean of rank $n$. Let $\phi:[H,G] \longrightarrow \mathcal{B}_n$ be the lattice isomorphism such that $\phi(K_i)=\{i\}$. Then $\forall X \in [H,G]$, $$K_i \nsubseteq X \Leftrightarrow \{i\} \nsubseteq \phi(X) \ \Leftrightarrow \phi(X) \subseteq {\{i\}}^\complement \  \Leftrightarrow X \subseteq {K_i}^\complement = M_i.$$   
% 
%
%\noindent Now assume that $g \in M_i$, then 
%$$ Y \wedge M_i = (X \vee K_i) \wedge M_i =(X \wedge M_i) \vee (K_i \wedge M_i)
%= X \vee H = X. $$

%The second equality uses the distributivity of $[H,G]$ Boolean, and the third follows from the fact that $X \subseteq M_i$ and $K_i \wedge M_i=K_i \wedge {K_i}^\complement=H$.
By the Boolean structure, $X = Y \wedge M_i$. This implies that $Xg=Yg \cap M_ig=Yg \cap M_i$ as $g \in M_i$. Conversely if $Xg=Yg \cap M_i$, then $g \in Xg \subseteq M_i$.
\end{proof}

\begin{example} \label{Z6} The edge labeling $el$ for $\hat{C}([0], C_6)$:   
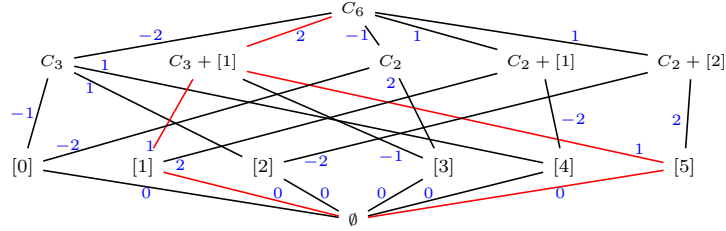
\begin{center} $\begin{tikzpicture}
\node (G) at (0,-1*.7) {\tiny $C_6$};
\node (A0) at (-4,-2*.7) {\tiny $C_3$};
\node (A1) at (-2,-2*.7) {\tiny $C_3+[1]$};
\node (B0) at (0.5,-2*.7) {\tiny $C_2$};
\node (B1) at (2.5,-2*.7) {\tiny $C_2 + [1]$};
\node (B2) at (4.5,-2*.7) {\tiny $C_2 + [2]$};
\node (H0) at (-5.5*0.8,-4*.7) {\scriptsize $[0]$};
\node (H1) at (-3.5*0.8,-4*.7) {\scriptsize $[1]$};
\node (H2) at (-1.5*0.8,-4*.7) {\scriptsize $[2]$};
\node (H3) at (1.5*0.8,-4*.7) {\scriptsize $[3]$};
\node (H4) at (3.5*0.8,-4*.7) {\scriptsize $[4]$};
\node (H5) at (5.5*0.8,-4*.7) {\scriptsize $[5]$};
\node (0) at (0,-5*.7) {\tiny $\emptyset$};
\node  at (-4*0.5-0.7,-1.5*.7) {\color{blue}{\tiny $-2$}};
\node  at (-2*0.5+0.3,-1.5*.7) {\color{blue}{\tiny $2$}};
\node  at (0.5*0.5-0.2,-1.5*.7) {\color{blue}{\tiny $-1$}};
\node  at (2.5*0.5-0.4,-1.5*.7) {\color{blue}{\tiny $1$}};
\node  at (4.5*0.5+0.7,-1.5*.7) {\color{blue}{\tiny $1$}};
\node  at (-5.5*0.5,-4.5*.7) {\color{blue}{\tiny $0$}};
\node  at (-2*0.5,-4.5*.7) {\color{blue}{\tiny $0$}};
\node  at (-.75*0.5,-4.5*.7) {\color{blue}{\tiny $0$}};
\node  at (.75*0.5,-4.5*.7) {\color{blue}{\tiny $0$}};
\node  at (2*0.5,-4.5*.7) {\color{blue}{\tiny $0$}};
\node  at (5.5*0.5,-4.5*.7) {\color{blue}{\tiny $0$}};
\node  at (-5.5*0.8,-3*.7) {\color{blue}{\tiny $-1$}};
\node  at (-3.8,-3.6*.7) {\color{blue}{\tiny $-2$}};  
\node  at (-2.7,-3.6*.7) {\color{blue}{\tiny $1$}};
\node  at (-2.3,-3.95*.7) {\color{blue}{\tiny $2$}};
\node  at (-3.5,-2.5*.7) {\color{blue}{\tiny $1$}}; 
\node  at (-3.3,-2.05*.7) {\color{blue}{\tiny $1$}}; 
\node  at (-0.5,-3.9*.7) {\color{blue}{\tiny $-2$}};
\node  at (0.5,-3.8*.7) {\color{blue}{\tiny $-1$}};
\node  at (3.8,-3.65*.7) {\color{blue}{\tiny $1$}}; 
\node  at (4.3,-3.1*.7) {\color{blue}{\tiny $2$}};
\node  at (2.92,-3.1*.7) {\color{blue}{\tiny $-2$}};
\node  at (0.5,-2.4*.7) {\color{blue}{\tiny $2$}};

\tikzstyle{segm}=[-,>=latex, semithick]
\draw [segm] (G)to(A0); \draw [segm, red] (G)to(A1);
\draw [segm] (G)to(B0);\draw [segm] (G)to(B1);\draw [segm] (G)to(B2);
\draw [segm] (A0)to(H0); \draw [segm] (A0)to(H2);\draw [segm] (A0)to(H4);
\draw [segm, red] (A1)to(H1); \draw [segm] (A1)to(H3);\draw [segm,red] (A1)to(H5);
\draw [segm] (B0)to(H0); \draw [segm] (B0)to(H3);
\draw [segm] (B1)to(H1); \draw [segm] (B1)to(H4);
\draw [segm] (B2)to(H2); \draw [segm] (B2)to(H5);
\draw [segm] (0)to(H0); \draw [segm, red] (0)to(H1);\draw [segm] (0)to(H2);
\draw [segm] (0)to(H3);\draw [segm] (0)to(H4);\draw [segm, red] (0)to(H5);
\end{tikzpicture} 
$ \end{center}
\end{example}

%Now we discuss a necessary and sufficient condition for the edge labeling $el$ (see Definition \ref{el}) to be a dual EL-labeling.

\begin{theorem}\label{EL}
Let $[H,G]$ be a Boolean interval. The edge labeling $el$ on $\hat{C}(H,G)$  is a dual EL-labeling if and only if $[H,G]$ is group-complemented.
\end{theorem}
\begin{proof}
Each interval in $\hat{C}(H,G)$ should have a unique strictly increasing maximal chain (from top to bottom) which is also lexicographically first.  We will consider two kinds of intervals individually:

\vspace{.2cm}
\noindent Case 1. $[L_1g,L_2g]$: \\ 
By the Boolean structure of $[H,G]$, we can write $L_2$ uniquely as $L_1 \vee (\vee_{i \in I}K_i)$ with $I = \{i_1, \dots, i_p\}$. So by Lemma \ref{posetisom}, any maximal chain in $[L_1g,L_2g]$ is of the form
\begin{equation*}
L_1g \lessdot (L_1 \vee K_{i_1})g \lessdot \ldots \lessdot (L_1 \vee (\vee_{\substack{i \in I}}K_i))g=L_2g.
\end{equation*}
The labeling of all such chains is same up to permutation, because the sign of a label $i$ does not depend on choice of the chain, but depends only on the fact that $g \in M_i$ or not, by Lemma \ref{ginmi}. Now, by the Boolean structure, every permutation occurs exactly once. We can choose the lexicographically first, which is then unique and strictly increasing.

\smallskip
\noindent Case 2. $[\emptyset,Lg]$:\\
Take $I=\{i_1, \ldots, i_p\}$ such that $L=\vee_{i \in I}K_i$. 
%We assume that $i_1< \ldots <i_p$. 
Using the Boolean structure of $[H,G]$, any maximal chain in $[\emptyset,Lg]$ is of the form (where $g' \in Lg$) 
\begin{equation*}
\emptyset \lessdot Hg' \lessdot K_{i_1}g' \lessdot \ldots \lessdot (\vee_{\substack{i \in I}}K_i)g'=Lg'.
\end{equation*} 
%\vspace{.2cm}
\noindent \emph{Existence of a strictly increasing chain}\\
\noindent \emph{- Necessary condition:}\\  
Since the label of the leftmost edge of the above chain is $0$, for the existence of a strictly increasing chain, it is necessary to find $g' \in Lg$ for which all the other labels are negative, which means that $g' \in \cap_{i \in I} M_i=L^{\complement}.$ So, 
\begin{itemize}
\item[] $[\emptyset,Lg]$ admits a strictly increasing maximal chain, 
\item[$\Leftrightarrow$] $L^{\complement} \cap Lg \neq \emptyset$,
\item[$\Leftrightarrow$] there exist $l \in L$ and $l' \in L^{\complement}$ such that $l'=lg$,
\item[$\Leftrightarrow$] $g \in LL^{\complement}.$
\end{itemize}  
As these equivalent conditions should hold $\forall g \in G$ and $\forall L \in [H,G]$, we conclude that $G=LL^{\complement}$ and $[H,G]$ is group-complemented. 

\vspace{.05cm}
\noindent \emph{- Sufficient condition:}\\
The existence of $g' \in L^{\complement} \cap Lg$ is sufficient because by Case 1, there exists a unique strictly increasing maximal chain in $[Hg',Lg']$ which by adding the last label $0$, is still strictly increasing on $[\emptyset,Lg]$.

\vspace{.2cm}
\noindent The chain built just above is also \emph{lexicographically first} on $[\emptyset,Lg]$. 

\vspace{.2cm}
\noindent \emph{Uniqueness of the strictly increasing chain} \\
\noindent For the uniqueness, we just need to show that there exists a unique possible $Hg'$. Let $g_1, g_2 \in L^{\complement} \cap Lg$, we have $g_1=l_1g$, $g_2=l_2g$ and hence $g_1{g_2}^{-1}=l_1{l_2}^{-1} \in L$. Moreover, $g_1{g_2}^{-1} \in L^{\complement}$. Therefore  $g_1{g_2}^{-1} \in L \cap L^{\complement}=H$. Thus, $Hg_1=Hg_2$.
\end{proof}

\begin{remark}
The group-complemented assumption is necessary and sufficient for the intervals $[\emptyset,Lg]$, but this works in general for those of the form $[L_1g,L_2g]$. 
\end{remark}

By Theorem \ref{ELCM} and Remark \ref{dualrem}, Theorem \ref{EL} has the following consequence.  

\begin{corollary} \label{cp1}
Let $[H,G]$ be a Boolean group-complemented interval. Then $\hat{C}(H,G)$ is Cohen-Macaulay.
\end{corollary}
%\begin{proof}
%Since the order complex is invariant by dual, so is the Cohen-Macaulay property. Therefore the result follows from Theorem \ref{EL} and Theorem \ref{ELCM}.
%\end{proof}

For better understanding of Theorem \ref{EL}, we discuss an easy example.  Since  the interval $[[0], C_6]$ is group-complemented, the labeling of Example \ref{Z6} is a dual EL-labeling. Moreover, there are two maximal decreasing chains, so by Theorem \ref{ELCM}, the order complex of the proper part of $\hat{C}([0], C_6)$  has the homotopy type of the wedge of two $\mathbb{S}^1$'s.

\begin{theorem} \label{cp2}
Let $[H,G]$ be a Boolean group-complemented interval. For $P=C(H,G)$, the nontrivial reduced Betti number of the order complex $\Delta(P)$ is exactly the dual Euler totient $\hat{\varphi}(H,G)$. It is then the number of cosets $Hg$ such that $\langle Hg \rangle = G$, which is nonzero.
\end{theorem}
\begin{proof}
Recall that $\mu(\hat{P})=\tilde{\chi}(\Delta(P))$. Therefore, by Corollary \ref{cp1}, $$\mu(\hat{P}) = (-1)^{\ell(P)}\beta_{\ell(P)}(\Delta(P)).$$ But $\ell(P)=\ell(H,G)-1$, so by Lemma \ref{mobeul0}, $ \beta_{\ell(P)}(\Delta(P)) = \hat{\varphi}(H,G).$
The last statement follows by Definition \ref{euler}, Lemma \ref{grpcplem} and Theorem \ref{ore2}. 
\end{proof}

\subsection{Examples}

Apart from the group-complemented intervals of the previous section, we will exhibit other classes of examples for which Question \ref{QCM} has a positive answer. The following result was pointed out to us by Woodroofe:

\begin{theorem}[\cite{bgs}, p14] \label{connl3}
A graded poset of length $1$ or $2$ is Cohen-Macaulay. A graded poset of length $3$ is Cohen-Macaulay if and only if its proper part is connected.
\end{theorem}  

\begin{proposition} \label{conn}
Let $[H,G]$ be a Boolean interval of finite groups. Then the proper part of $\hat{C}(H,G)$ is connected.
\end{proposition}  
\begin{proof}
Let $Hg$ and $Hg'$ be two atoms of $\hat{C}(H,G)$. Let $K_1, \dots, K_n$ be the atoms of $[H,G]$. Then by the Boolean structure, $g''=  g'g^{-1}$ is a product of elements in some $K_i$, i.e. $ g''=k_{\tau(1)} k_{\tau(2)} \dots k_{\tau(s)} $ with $\tau(i) \in \{ 1,2, \dots , n\}$ and $k_{\tau(i)} \in K_{\tau(i)}$.   

Now, $g' = g''g$, so we get that $ g'=k_{\tau(1)} k_{\tau(2)} \dots k_{\tau(s)}g $.
Let $g_i= k_{\tau(s-i)} \dots k_{\tau(s)}g$. Then $Hg_i$ and $Hg_{i+1}$ are connected via $K_{\tau(s-i-1)}g_i$. We deduce that $Hg$ and $Hg'$ are connected. But any element of the proper part of $\hat{C}(H,G)$ is connected to an atom.
\end{proof} 

\begin{lemma}\label{r2to} Let $[H,G]$ be a rank $2$ Boolean interval of finite groups. Then the dual Euler totient $\hat{\varphi}(H,G) \geq 1$.
\end{lemma}
\begin{proof}
Let $K_1$, $K_2$ be the atoms of $[H,G]$. Then $$\hat{\varphi}(H,G) = |G:H|-|G:K_1|-|G:K_2|+|G:G|.$$ But $|G:H| = |G:K_i| \cdot |K_i:H| \geq 2|G:K_i| $. Thus $\hat{\varphi}(H,G) \geq |G:G| = 1$.
\end{proof}

\begin{corollary} \label{r2cm} The graded coset poset of a rank $2$ Boolean interval $[H,G]$ is Cohen-Macaulay and the nontrivial reduced Betti number is nonzero.
\end{corollary} 
\begin{proof} 
By Proposition \ref{conn}, Lemma \ref{r2to}, Theorems \ref{connl3} and \ref{ELCM}.
\end{proof}

\begin{remark} \label{rkcp}
If $|G:H|<32$, then there are, up to equivalence, $612$ Boolean intervals (i.e. $1 + 241 + 337 + 33$, according to the rank $0,1,2,3$), by GAP. They are all group-complemented except $[D_8,A_2(2)]$ and $[S_3,A_2(2)]$, both of rank $2$.
\end{remark}

\begin{corollary} The graded coset poset of a Boolean interval $[H,G]$ of index  $<32$ is Cohen-Macaulay and the nontrivial reduced Betti number is nonzero.
\end{corollary}  
\begin{proof}
By Remark \ref{rkcp}, Corollary \ref{cp1}, Theorem \ref{cp2} and Corollary \ref{r2cm}.
\end{proof}
In the following example, for the notions of BN-pair, Coxeter system, Borel subgroup, spherical building, simple groups of Lie type, Chevalley groups and Dynkin diagram, we refer to the books \cites{bou,bro,gar,car}.

\begin{example} \label{exBN}  Let $G$ be a finite group with a BN-pair, $H$ be the corresponding Borel subgroup and $(W,S)$ be the associated Coxeter system. Let $n:=|S|$ be the rank of the BN-pair. Then the interval $[H,G]$ is Boolean of rank $n$ \cite[Theorem 8.3.4]{car}. The order complex of $C(H,G)$ is equivalent to the spherical building associated to the BN-pair \cite[Section 5.7]{gar} as abstract simplicial complexes. It is Cohen-Macaulay (\cite{bjo},\cite[Remark 3 p94]{bro}) and has the homotopy type of a wedge of $r( \ge 1)$ spheres $\mathbb{S}^{n-1}$ \cite[Theorem 2 p93]{bro}. It follows that the dual Euler totient $\hat{\varphi}(H,G) = r \neq 0$, so by Theorem \ref{main}, the Boolean interval $[H,G]$ is linearly primitive. Any finite simple group $G$ of Lie type (over a finite field of characteristic $p$) admits a BN-pair (except Tits group) and $r$ is the $p$-contribution in the order of $G$ \cite[Section 4, (ii')]{tits}. If moreover, $G$ is a Chevalley group, then $n$ is the number of vertices in its Dynkin diagram.  
\end{example}

The rank $2$ Boolean intervals $[D_8,A_2(2)]$ and $[S_3,A_2(2)]$ of Remark \ref{rkcp}, have dual Euler totients  $\hat{\varphi} = 2^3$ and $15$, respectively. The first comes from a BN-pair, but not the second.

\begin{remark} With the help of Hulpke and GAP, four Boolean intervals $[H,G]$ of rank $3$, with $G$ simple and $|G| \le 4 \cdot 10^6$ (given by $G={\rm A}_3(2)$, ${\rm C}_3(2)$, $^2{\rm A}_2(5^2)$  and $^2{\rm A}_3(3^2)$; all of Lie type) are found. None of them is group-complemented. Their corresponding dual Euler totients $\hat{\varphi}$ are $2^6$, $2^9$, $3899$ and $3968$, respectively. The first two come from BN-pairs, but not the two last (because $3899$ and $3968$ are not prime-powers). Using SageMath \cite{sage}, one can check that the coset poset of the third is also Cohen-Macaulay (we don't know about the last one).
\end{remark}

\section{Acknowledgments} 
A special thanks to Russ Woodroofe for fruitful exchanges. Thanks to the anonymous referees for useful suggestions. Thanks to Viakalathur Shankar Sunder, Vijay Kodiyalam and Amritanshu Prasad for their interest to our works. Thanks also to Darij Grinberg, Todd Trimble, Brian M. Scott and Alexander Hulpke. Mamta Balodi thanks Sebastien Palcoux for introducing her to the subject of planar algebras and Ore's theorem for cyclic subfactor planar algebras. This work was supported by the Institute of Mathematical Sciences, Chennai. Mamta Balodi is now supported by Department of Science and Technology, Govt of India, New Delhi, under the grant PDF/2017/000229.

\end{document}